\documentclass[12pt]{amsart}

\usepackage[left=2cm, right=2cm, top=3cm, bottom=3cm]{geometry}

\usepackage{amsmath, amsthm, amsfonts, amssymb}
\usepackage{mathrsfs}
\usepackage{mathtools}
\usepackage{upgreek}
\usepackage{longtable}
\usepackage{booktabs}
\usepackage{hyperref}

\makeatletter
\@namedef{subjclassname@2020}{%
  \textup{2020} Mathematics Subject Classification}
\makeatother

\theoremstyle{plain}
\newtheorem{thm}{Theorem}[section]
\newtheorem*{thm*}{Main Theorem, Part I}{}
\newtheorem*{thm2*}{Main Theorem, Part II}{}

\newtheorem{prop}[thm]{Proposition}

\newtheorem{lemm}[thm]{Lemma}

\theoremstyle{definition}

\theoremstyle{remark}

\newcommand\twoscript[2]{\substack{{#1} \\ {#2}}}

\newcommand\rmi{\mathrm{i}}

\newcommand\legendre[2]{\genfrac{(}{)}{}{}{#1}{#2}}
\newcommand\tbtmat[4]{\left(\begin{smallmatrix}{#1} & {#2} \\ {#3} & {#4}\end{smallmatrix}\right)}
\newcommand\tbtMat[4]{\begin{pmatrix}{#1} & {#2} \\ {#3} & {#4}\end{pmatrix}}
\newcommand*\abs[1]{\lvert#1\rvert}
\newcommand\etp[1]{\mathfrak{e}\left(#1\right)}

\newcommand\rad{\mathop{\mathrm{rad}}}

\newcommand\numZ{\mathbb{Z}}

\newcommand\numC{\mathbb{C}}

\newcommand\numgeq[2]{\mathbb{#1}_{\geq #2}}

\newcommand\slZ{\mathrm{SL}_2(\mathbb{Z})}

\newcommand\uhp{\mathfrak{H}}

\allowdisplaybreaks
\begin{document}


\baselineskip=17pt


\title[Fourier coefficients of eta-quotients]{A formula for Fourier coefficients of certain eta-quotients and their expansions as Eisenstein series}

\author{Xiao-Jie Zhu}
\address{School of Mathematical Sciences\\
Key Laboratory of MEA (Ministry of Education) \& Shanghai Key Laboratory of PMMP\\
East China Normal University\\
500 Dongchuan Road, 200241\\
Shanghai, P. R. China}
\email{zhuxiaojiemath@outlook.com}
\urladdr{https://orcid.org/0000-0002-6733-0755}

\date{}

\begin{abstract}
We give a list of $113$ holomorphic eta-quotients of integral weight ($66$ of which are primitive) and provide a uniform closed formula for their Fourier coefficients $c(l)$ where $l\equiv1\bmod{m}$ with some fixed $m\mid24$. The proof involves Wohlfahrt's extension of Hecke operators and a dimension formula for spaces of modular forms of general multiplier system. We further provide the expansions of these eta-quotients as linear combinations of standard Eisenstein series.
\end{abstract}



\maketitle


\section{Introduction}
\label{sec:Introduction}
Let $N$ be a positive integer and for each positive divisor $n$ of $N$ let $r_n$ be an integer. Set $\mathbf{r}:=(r_n)_{n\mid N}$. In this paper, we study the function
\begin{equation*}
f(\tau)=f_{\mathbf{r}}(\tau):=\prod_{n\mid N}\prod_{m=1}^{+\infty}(1-q^{nm})^{r_n},
\end{equation*}
where $q=\exp2\uppi\rmi\tau$ and $\tau\in\uhp:=\{z\in\numC\colon\Im z>0\}$. We only consider the situation $\sum_{n\mid N}n\cdot r_n=0$ and $\sum_{n\mid N}r_n\equiv0\pmod{2}$, in which $f$ can be written as an integral-weight eta-quotient $f(\tau)=\prod_{n\mid N}\eta(n\tau)^{r_n}$ with $\eta(\tau)=q^{1/24}\prod_{m=1}^{+\infty}(1-q^{m})$ being the Dedekind eta function (cf. \cite[Section 1]{Bha21}).

To state the first part of our main theorem, we need the following symbols:
\begin{equation}
\label{eq:kvarPimr}
k=\frac{1}{2}\sum_{n\mid N}r_n,\quad\varPi=\prod_{n\mid N}(Nn^{-1})^{\abs{r_n}},\quad m_\mathbf{r}=\frac{24}{\gcd(24,\sum_{n\mid N}nr_n,\sum_{n\mid N}Nn^{-1}r_n)}.
\end{equation}
In addition, we define $\delta=0$ (or $\delta=1$ respectively) if $\varPi$ takes the form $2^\alpha\cdot(4m+1)$ (or $2^\alpha\cdot(4m+3)$ respectively) with $\alpha,\,m\in\numgeq{Z}{0}$. Finally, $\legendre{a}{b}$ refers to the Kronecker-Jacobi symbol.

\begin{thm*}
Let $(N, \mathbf{r})$ be a pair listed in Table \ref{table:etaQuotients} (cf. Appendix \ref{sec:The table}) and let the Fourier expansion be $f_{\mathbf{r}}(\tau)=\sum_{n=0}^{+\infty}c_f(n)q^n$. Then for any positive integer $l$ with $l\equiv1\pmod{m_{\mathbf{r}}}$ we have
\begin{equation}
\label{eq:main}
c_f(l)=-r_1\cdot\sum_{\twoscript{a\mid l}{\gcd(a,N)=1}}\legendre{a}{\varPi}\cdot a^{k-1}\cdot\varepsilon_{l,a},
\end{equation}
where\footnote{The symbol $\varepsilon_{l,a}$ depends on $N$ and $\mathbf{r}$ as well.} $\varepsilon_{l,a}\in\{\pm1\}$ is defined as
\begin{equation*}
\varepsilon_{l,a}:=\begin{dcases}
1 & \text{ if } 2\mid l,\,2\nmid N,\\
(-1)^{\frac{(k+\delta)(a-1)}{2}} &\text{ otherwise}.
\end{dcases}
\end{equation*}
\end{thm*}

The functions $f_\mathbf{r}$ corresponding to the entries $(N, \mathbf{r})$ in Table \ref{table:etaQuotients} are exactly those integral-weight holomorphic eta-quotients nonvanishing at infinity and in one-dimensional modular form spaces whose dimensions can be computed by \cite[Theorem 4.2]{Zhu24}. Thus, to prove the above theorem, we use Wohlfahrt's extension of Hecke operators $T_l$ acting on $f_\mathbf{r}$, where $l\equiv1\pmod{m_{\mathbf{r}}}$. This leads to identities $T_lf_\mathbf{r}=c_l\cdot f_\mathbf{r}$ and analyzing these gives \eqref{eq:main}. The details are presented in Section \ref{sec:proof}.

As the second part of our main theorem, we provide the expansion of each eta-quotient in Table \ref{table:etaQuotients} as a linear combination of standard Eisenstein series. Let $B_m$ be the $m$-th Bernoulli number. For $k\in\numgeq{Z}{2}$, $t\in\numgeq{Z}{1}$ and $D_1$, $D_2$ being fundamental discriminants\footnote{An integer $d\neq0$ is called a fundamental discriminant if either $d\equiv1\bmod{4}$ and $d$ is square-free, or $4\mid d$ and $d/4\equiv2,3\bmod{4}$ is square-free. It is well-known that the function $n\mapsto\legendre{d}{n}$ is a primitive Dirichlet character modulo $\abs{d}$.}, if $(k,D_1,D_2)\neq(2,1,1)$, set
\begin{equation*}
E_{k}^{D_1,D_2,t}(\tau)=\delta_{\abs{D_1},1}\cdot L\left(1-k,\legendre{D_2}{\cdot}\right)+2\sum_{n=1}^{+\infty}\sigma_{k-1}^{D_1,D_2}(n)q^{tn},
\end{equation*}
where $\delta_{\abs{D_1},1}=1$ if $\abs{D_1}=1$ and $\delta_{\abs{D_1},1}=0$ otherwise, and where the special $L$-value is given by
\begin{equation*}
L\left(-s,\chi\right):=-\frac{1}{s+1}\sum_{m=0}^{s+1}\binom{s+1}{m}B_mM^{m-1}\sum_{r=1}^{M}\chi(r)r^{s+1-m},\quad s\in\numgeq{Z}{0}
\end{equation*}
for primitive Dirichlet character $\chi$ modulo $M$, and where
\begin{equation*}
\sigma_{k-1}^{D_1,D_2}(n):=\sum_{0<d\mid n}\legendre{D_1}{n/d}\legendre{D_2}{d}d^{k-1}
\end{equation*}
is the generalized power sum. If $(k,D_1,D_2)=(2,1,1)$ then we set
\begin{equation*}
E_{2}^{1,1,t}(\tau)=E_2^{1,1}(\tau)-tE_2^{1,1}(t\tau)
\end{equation*}
where
\begin{equation*}
E_2^{1,1}(\tau)=-\frac{1}{12}+2\sum_{n=1}^{+\infty}\sigma_{1}^{1,1}(n)q^n.
\end{equation*}
Finally, if $k=1$, we set
\begin{equation*}
E_{1}^{D_1,D_2,t}(\tau)=\delta_{\abs{D_1},1}\cdot L\left(0,\legendre{D_2}{\cdot}\right)+\delta_{\abs{D_2},1}\cdot L\left(0,\legendre{D_1}{\cdot}\right)+2\sum_{n=1}^{+\infty}\sigma_{0}^{D_1,D_2}(n)q^{tn}.
\end{equation*}
See \cite[\S 4.5, 4.6 and 4.8]{DS05} for more general definitions and basic properties of $E_{k}^{D_1,D_2,t}$.

\begin{thm2*}
Let $(N, \mathbf{r})$ be a pair listed in Table \ref{table:etaQuotients} (cf. Appendix \ref{sec:The table}) whose weight is $k$. Then $f_{\mathbf{r}}(\tau)$ can be uniquely expressed as a linear combination of the functions $E_{k}^{D_1,D_2,t}$ where $D_1$, $D_2$ are fundamental discriminants and $t\in\numgeq{Z}{1}$ such that $D_1D_2t\mid24N$ (and if $k=1$ such that $D_1\leq D_2$). The expressions are presented after Table \ref{table:etaQuotients}.
\end{thm2*}
The proof is a brute force verification. For each identity, we verify that the coefficients of the $q^n$-terms of both sides are the same for $n\leq n_0$ with $n_0$ sufficiently large with the help of SageMath programs and then use the Sturm bound (c.f., e.g., \cite[Coro. 5.6.14]{CS17}) to ensure the validity. The details are given in Section \ref{sec:proof2}.

Here is an example:
\begin{multline}
\label{eq:exampetaEis}
\eta(\tau)^{-1}\eta(2\tau)^2\eta(3\tau)\eta(6\tau)^{-3}\eta(12\tau)^5\eta(24\tau)^{-2}\\
=\frac{1}{4}E_1^{-3,8,1}+\frac{1}{2}E_1^{-3,8,2}+\frac{3}{4}E_1^{-3,8,3}+\frac{3}{2}E_1^{-3,8,6}+\frac{1}{4}E_1^{-24,1,1}-\frac{1}{2}E_1^{-24,1,2}-\frac{3}{4}E_1^{-24,1,3}+\frac{3}{2}E_1^{-24,1,6}.
\end{multline}

The $113$ functions studied in this note are also known as products with multiplicative series in the theory of basic hypergeometric series. Fine presented many such functions in \cite[\S 33]{Fin88}, some of which overlap ours. Since then there are some works establishing identities between eta-quotients and Eisenstein series; c.f., e.g, \cite{Wil12,YXJ13,MW24}.  Martin \cite{Mar96} obtained the complete list of all holomorphic eta-quotients that are simultaneously common Hecke eigenforms. The only functions in Table \ref{table:etaQuotients} that overlap Martin's table are $1^{-4}2^{10}4^{-4}$ and $1^{-2}2^{3}4^{3}8^{-2}$. It seems strange at first glance since \eqref{eq:Tlfclf} holds for all possible $l$ and all $f_{\mathbf{r}}$ in Table \ref{table:etaQuotients}, namely, $f_{\mathbf{r}}$ are also ``common Hecke eigenfunctions'' but are excluded from Martin's list. We explain this ``paradox'' in the third paragraph of Section \ref{sec:remarks}.

In this note, we use the computer algebra system SageMath \cite{Sage} constantly. For a nice account of eta-quotients, modular forms and related topics, see \cite{Ono04,DS05,KG11,CS17}. There are many excellent works concerned with eta-quotients and their relations with Eisenstein series and theta series; c.f., e.g., \cite{GH93,Mar96,LO13,Bha17,Oga18,CKL19,AI23,ZZ23}.

\section{Prerequisites}
\label{sec:Prerequisites}
First, we describe the dimension formula we need. For any positive integer $N$, the congruence subgroup $\Gamma_0(N)$ is the group of unimodular integral matrices $\tbtmat{a}{b}{c}{d}$ with $N\mid c$. Set $\slZ:=\Gamma_0(1)$ and
\begin{align}
m&=[\slZ\colon\Gamma_0(N)]=N\prod_{p\mid N}\left(1+\frac{1}{p}\right),\label{eq:mofGamma0N}\\
\varepsilon_2&=\begin{dcases}
\prod_{p \mid N}\left(1+\legendre{-4}{p}\right) & \text{ if } 4\nmid N,\\
0 & \text{ if } 4\mid N,
\end{dcases}\notag\\
\varepsilon_3&=\begin{dcases}
\prod_{p \mid N}\left(1+\legendre{-3}{p}\right) & \text{ if } 9\nmid N,\\
0 & \text{ if } 9\mid N.
\end{dcases}\notag
\end{align}

\begin{thm}
\label{thm:dimension}
Let $r_n$ be an integer for any $n\mid N$ and set $\mathbf{r}=(r_n)_{n\mid N}$. If $k>2-\frac{6}{m}\varepsilon_2-\frac{8}{m}\varepsilon_3-\frac{12}{m}\sum_{c\mid N}\phi(c,N/c)\cdot\left(1-\left\{\frac{x_c}{24}\right\}\right)$, then
\begin{equation}
\label{eq:dimension}
\dim_\numC M_k(\Gamma_0(N), \chi_\mathbf{r})=\frac{k-1}{12}m+\frac{1}{4}\varepsilon_2+\frac{1}{3}\varepsilon_3+\sum_{c\mid N}\phi(c,N/c)\cdot\left(\frac{1}{2}-\left\{\frac{x_c}{24}\right\}\right),
\end{equation}
where
\begin{equation*}
k=\frac{1}{2}\sum_{n\mid N}r_n,\quad x_c=\sum_{n\mid N}\frac{N}{(N,c^2)}\cdot\frac{(n,c)^2}{n}r_n\quad\text{  for }c\mid N.
\end{equation*}
\end{thm}
Note that the notation $\phi(c,N/c)$ refers to the Euler totient $\phi$ of the greatest common divisor of $c$ and $N/c$. The notation $(N,c^2)=\gcd(N,c^2)$ refers to the greatest common divisor. The notation $\left\{\frac{x_c}{24}\right\}$ means the fractional part of $\frac{x_c}{24}$. The space $M_k(\Gamma_0(N), \chi_\mathbf{r})$ refers to that consisting of modular forms of weight $k$, on group $\Gamma_0(N)$ and with multiplier system $\chi_\mathbf{r}$, which is the multiplier system of the eta-quotient $\prod_{n\mid N}\eta(n\tau)^{r_n}$ (c.f. \cite[eq. (16)]{ZZ23} or \cite[eq. (6)]{Zhu24}). For a proof of Theorem \ref{thm:dimension}, see \cite[Section 4]{Zhu24}.

As one can verify one by one, if $f_\mathbf{r}$ is the function corresponding to an entry $(N, \mathbf{r})$ in Table \ref{table:etaQuotients}, then $k\in\numZ$, $f_\mathbf{r}\in M_k(\Gamma_0(N), \chi_\mathbf{r})$, $\dim_\numC M_k(\Gamma_0(N), \chi_\mathbf{r})$ can be computed by Theorem \ref{thm:dimension} and $\dim_\numC M_k(\Gamma_0(N), \chi_\mathbf{r})=1$. Indeed, they exhaust\footnote{However, there may exist holomorphic eta-quotients $f_\mathbf{r}$ such that $f_\mathbf{r}\in M_k(\Gamma_0(N), \chi_\mathbf{r})$, $\dim_\numC M_k(\Gamma_0(N), \chi_\mathbf{r})=1$ but $\dim_\numC M_k(\Gamma_0(N), \chi_\mathbf{r})$ \emph{cannot} be computed by Theorem \ref{thm:dimension}, i.e., the inequality on $k$ does not hold.} all holomorphic eta-quotients with these properties as one can check using any computer algebra system based on the algorithm presented in \cite[Remark 9.4]{Zhu24}. See the last paragraph of Section \ref{sec:remarks}.

Then, we describe another tool---Wohlfahrt's extension of Hecke operators. The original reference is \cite{Woh57}. The following description is due to the author \cite[Theorem 9.15]{Zhu24}.
\begin{thm}
\label{thm:cTlfbycf}
Let $N\in\numgeq{Z}{1}$. For each $n\mid N$, let $r_n$ be an integer and set 
\begin{equation*}
\mathbf{r}=(r_n)_{n\mid N},\quad k'=\frac{1}{2}\sum_{n\mid N}r_n,\quad x_N=\sum_{n\mid N}nr_n,\quad \varPi=\prod_{n\mid N}(Nn^{-1})^{\abs{r_n}}.
\end{equation*}
Let $\chi_\mathbf{r}$ be the multiplier system of $\prod_{n\mid N}\eta(n\tau)^{r_n}$. Let $f$ be a meromorphic modular form (assumed to be holomorphic on $\uhp$) on $\Gamma_0(N)$ of weight $k\in k'+2\numZ$ with multiplier system $\chi_{\mathbf{r}}$ and let $l$ be a positive integer with $l\equiv1\pmod{m_{\mathbf{r}}}$. Then
\begin{enumerate}
  \item[(a)] we have the expansions
  \begin{equation*}
  f(\tau)=\sum_{n\in\frac{x_N}{24}+\numZ}c_f(n)q^n,\quad T_lf(\tau)=\sum_{n\in\frac{x_N}{24}+\numZ}c_{T_lf}(n)q^n,
  \end{equation*}
  where $c_f(n),\,c_{T_lf}(n)\in\numC$ are uniquely determined complex numbers and $T_l$ is the operator defined in \cite[Section 9.3]{Zhu24};
  \item[(b)] for $n\in\frac{x_N}{24}+\numZ$ we have
  \begin{equation}
  \label{eq:cTlfbycf}
  c_{T_lf}(n)=l^{-\frac{k}{2}}\sum_{\twoscript{a\mid l,\,d=l/a}{(a,N)=1}}\legendre{a}{\varPi}a^k\cdot c_f\left(\frac{ln}{a^2}\right)\sum_{\twoscript{0\leq b<d}{(a,b,d)=1}}\legendre{-Nb}{(a,d)}^{2k}\etp{bd\left(\frac{n}{l}-\frac{x_N}{24}\right)}\psi_{l,\mathbf{r}}(a,b)
  \end{equation}
  where
  \begin{equation*}
  \psi_{l,\mathbf{r}}(a,b)=\begin{dcases}
  \etp{-\frac{k+\delta}{4}(d-1)+\frac{(k+\delta)(l-1)(N-1)}{4}} &\text{ if } 2\nmid l,\\
  \etp{-\frac{k+\delta}{4}(a-1)-\frac{N(k+\delta)(1+\delta_1)}{4}b} &\text{ if } 2\mid l,\,2\mid N,\\
  1 &\text{ if } 2\mid l,\,2\nmid N,
  \end{dcases}
  \end{equation*}
  with $\delta_1=1$ if $2\mid l$, $4\mid N$, $k\in\frac{1}{2}+\numZ$, $v_2(\varPi)\equiv1\bmod{2}$ and $\delta_1=0$ otherwise. ($\delta$ has been given in Introduction.)
\end{enumerate}
\end{thm}
Note that $\legendre{a}{\varPi}$ and $\legendre{-Nb}{(a,d)}$ are Kronecker-Jacobi symbols, not fractions, and that we set $c_f\left(\frac{ln}{a^2}\right)=0$ if $\frac{ln}{a^2}\not\in\frac{x_N}{24}+\numZ$. The notation $v_2(\varPi)$ refers to the $2$-adic exponential valuation and $\etp{x}:=\exp2\uppi\rmi x$. For a proof of this theorem, see \cite[Theorem 9.15]{Zhu24}.

The right-hand side of \eqref{eq:cTlfbycf} can be further simplified by working out the sum over $b$. We need such a simplification for $k\in\numZ$:
\begin{prop}
In Theorem \ref{thm:cTlfbycf} suppose $k\in\numZ$ (equivalently, $k'\in\numZ$), and
\begin{equation}
\label{eq:conditionlNdelta}
2\mid l,\,2\mid N,\, N(k+\delta)\equiv2\bmod{4} \text{ do not simultaneously hold.}
\end{equation}
Then for $n\in\frac{x_N}{24}+\numZ$ we have $n-\frac{lx_N}{24}\in\numZ$ and
\begin{equation}
\label{eq:cTlfbycfkIntegral}
c_{T_lf}(n)=l^{1-\frac{k}{2}}\sum_{\twoscript{a\mid l,\,d=l/a}{(a,N)=1}}\legendre{a}{\varPi}a^{k-1}\psi_{l,\mathbf{r}}(a,1) c_f\left(\frac{ln}{a^2}\right)\cdot\sum_{\twoscript{t\mid(a,d)}{(a/t)\mid n-lx_N/24}}\frac{\mu(t)}{t}.
\end{equation}
In particular, if $\rad(n-\frac{lx_N}{24},l)\mid N$, then $\frac{ln}{a^2}\in\frac{x_N}{24}+\numZ$ for any $a^2\mid l$ with $(a,N)=1$, and
\begin{equation*}
c_{T_lf}(n)=l^{1-\frac{k}{2}}\sum_{\twoscript{a^2\mid l,\,a>0}{(a,N)=1}}\legendre{a}{\varPi}\mu(a)a^{k-2}\psi_{l,\mathbf{r}}(a,1) c_f\left(\frac{ln}{a^2}\right).
\end{equation*}
\end{prop}
In the above proposition, $\mu(x)$ denotes the M\"obius function, that is, $\mu(p_1p_2\cdots p_t)=(-1)^t$ if $p_1, \dots,p_t$ are distinct primes and $\mu(x)=0$ is $x$ has nontrivial square factors, and $\rad(a,b):=\prod_{p\mid \gcd(a,b)}p$ where $p$ denotes a prime.
\begin{proof}
For $n\in\frac{x_N}{24}+\numZ$, we have $n-\frac{lx_N}{24}=n-\frac{x_N}{24}-\frac{(l-1)x_N}{24}\in\numZ$ since $24\mid (l-1)x_N$ by $l\equiv1\pmod{m_{\mathbf{r}}}$. For \eqref{eq:cTlfbycfkIntegral} there are three cases:
\begin{equation*}
\text{(a) } 2\nmid l,\qquad \text{(b) }2\mid l,\,2\mid N,\qquad \text{(c) }2\mid l,\,2\nmid N.
\end{equation*}
The proof of these cases are similar so we only present here that of case (b). By \eqref{eq:conditionlNdelta}, $\psi_{l,\mathbf{r}}(a,b)$ is independent of $b$ and hence $\psi_{l,\mathbf{r}}(a,b)=\psi_{l,\mathbf{r}}(a,1)$. Since $k\in\numZ$,
\begin{equation}
\label{eq:sumoverb}
\sum_{\twoscript{0\leq b<d}{(a,b,d)=1}}\legendre{-Nb}{(a,d)}^{2k}\etp{bd\left(\frac{n}{l}-\frac{x_N}{24}\right)}\psi_{l,\mathbf{r}}(a,b)=\psi_{l,\mathbf{r}}(a,1)\cdot\sum_{\twoscript{0\leq b<d}{(a,b,d)=1}}\etp{bd\left(\frac{n}{l}-\frac{x_N}{24}\right)}.
\end{equation}
Set $s=\frac{n}{l}-\frac{x_N}{24}$. Noting that $\sum_{t\mid(a,b,d)}\mu(t)=1$ if $(a,b,d)=1$ and $\sum_{t\mid(a,b,d)}\mu(t)=0$ otherwise, we have
\begin{align}
\sum_{\twoscript{0\leq b<d}{(a,b,d)=1}}\etp{bds}&=\sum_{0\leq b<d}\sum_{t\mid(a,b,d)}\mu(t)\etp{bds}\notag\\
&=\sum_{\twoscript{t\mid(a,d)}{tds\in\numZ}}\frac{\mu(t)d}{t}+\sum_{\twoscript{t\mid(a,d)}{tds\not\in\numZ}}\frac{1-\etp{d^2s}}{1-\etp{tds}}.\label{eq:sumovert}
\end{align}
We will prove $1-\etp{d^2s}=0$, which will imply the second sum in the last line vanishes. Indeed,
\begin{equation*}
d^2s=\left(\frac{ln}{a^2}-\frac{x_N}{24}\right)-\frac{(d^2-1)x_N}{24}.
\end{equation*}
We may assume at the beginning that $\frac{ln}{a^2}-\frac{x_N}{24}\in\numZ$ for otherwise $c_f\left(\frac{ln}{a^2}\right)=0$. Since $2\mid l$, $m_\mathbf{r}=1$ or $3$. If $m_\mathbf{r}=1$ then $24\mid x_N$ and hence $\frac{(d^2-1)x_N}{24}\in\numZ$. If $m_\mathbf{r}=3$, then $8\mid x_N$ and $3\nmid l$. Hence $3\nmid d$, which implies $3\mid d^2-1$. Again $\frac{(d^2-1)x_N}{24}\in\numZ$. This concludes the proof of $d^2s\in\numZ$, and hence $1-\etp{d^2s}=0$. Now inserting \eqref{eq:sumovert}, \eqref{eq:sumoverb} into \eqref{eq:cTlfbycf} we obtain \eqref{eq:cTlfbycfkIntegral}. The proof of the assertion under the condition $\rad(n-\frac{lx_N}{24},l)\mid N$ is omitted.
\end{proof}

\section{The proof of Main Theorem, Part I}
\label{sec:proof}
\begin{lemm}
Let $(N, \mathbf{r})$ be a pair listed in Table \ref{table:etaQuotients} (cf. Appendix \ref{sec:The table}). Then \eqref{eq:conditionlNdelta} is satisfied.
\end{lemm}
\begin{proof}
If $2\mid l$ and $2\mid N$, then $m_\mathbf{r}=1$ or $3$. Thus, by looking at the entries with $2\mid N$ and $m_\mathbf{r}=1$ or $3$ in Table \ref{table:etaQuotients} one by one, we find that $4\mid N(k+\delta)$, hence \eqref{eq:conditionlNdelta} is true.
\end{proof}
The following lemma is key to the proof of the first part of the main theorem.
\begin{lemm}
Let the notations and assumptions be as in the main theorem. We have
\begin{equation}
\label{eq:recursivecf}
\sum_{\twoscript{a^2\mid l,\,a>0}{(a,N)=1}}\legendre{a}{\varPi}\mu(a)a^{k-2}\psi_{l,\mathbf{r}}(a,1) c_f\left(\frac{l}{a^2}\right)
=-r_1\cdot\sum_{\twoscript{a\mid l}{(a,N)=1}}\legendre{a}{\varPi}a^{k-1}\psi_{l,\mathbf{r}}(a,1)\cdot\sum_{t\mid(a,l/a)}\frac{\mu(t)}{t}.
\end{equation}
\end{lemm}
\begin{proof}
As explained in the paragraphs following Theorem \ref{thm:dimension}, $f_\mathbf{r}\in M_k(\Gamma_0(N), \chi_\mathbf{r})$ and $\dim_\numC M_k(\Gamma_0(N), \chi_\mathbf{r})=1$. By \cite[Theorem 9.9]{Zhu24}, $T_l$ maps $M_k(\Gamma_0(N), \chi_\mathbf{r})$ into itself and hence
\begin{equation}
\label{eq:Tlfclf}
T_lf_\mathbf{r}=c_l\cdot f_\mathbf{r}
\end{equation}
with some $c_l\in\numC$. Equivalently, $c_{T_lf}(n)=c_lc_f(n)$ for $n\in\frac{x_N}{24}+\numZ$. Since $x_N=\sum_{j\mid N}j\cdot r_j=0$ we have $n\in\numZ$. Setting $n=0$ and noting $c_f(0)=1$ we find, according to \eqref{eq:cTlfbycfkIntegral}, that $l^{\frac{k}{2}-1}c_l$ equals the sum over $a$ in the right-hand side of \eqref{eq:recursivecf}. Now setting $n=1$, noting $c_f(1)=-r_1$ and using \eqref{eq:cTlfbycfkIntegral} again we arrive at \eqref{eq:recursivecf}.
\end{proof}

The first part of the main theorem now follows from this lemma and induction:
\begin{proof}[Proof of Main Theorem, Part I]
If for any prime $p\nmid N$ we have $p^2\nmid l$, then \eqref{eq:recursivecf} implies that
\begin{equation*}
c_f(l)=-r_1\cdot\sum_{\twoscript{a\mid l}{(a,N)=1}}\legendre{a}{\varPi}a^{k-1}\frac{\psi_{l,\mathbf{r}}(a,1)}{\psi_{l,\mathbf{r}}(1,1)}.
\end{equation*}
It is immediate that $\frac{\psi_{l,\mathbf{r}}(a,1)}{\psi_{l,\mathbf{r}}(1,1)}=\varepsilon_{l,a}$ and hence \eqref{eq:main} follows.

Now we use induction on the number of factors of $l$. In the induction step, we have, according to \eqref{eq:recursivecf}, that
\begin{equation}
\label{eq:recursivecf2}
c_f(l)=-r_1\cdot\sum_{\twoscript{a\mid l}{(a,N)=1}}\legendre{a}{\varPi}a^{k-1}\frac{\psi_{l,\mathbf{r}}(a,1)}{\psi_{l,\mathbf{r}}(1,1)}\cdot\sum_{t\mid(a,l/a)}\frac{\mu(t)}{t}-\sum_{\twoscript{a^2\mid l,\,a>1}{(a,N)=1}}\legendre{a}{\varPi}\mu(a)a^{k-2}\frac{\psi_{l,\mathbf{r}}(a,1)}{\psi_{l,\mathbf{r}}(1,1)} c_f\left(\frac{l}{a^2}\right).
\end{equation}
We need to prove that $l/a^2\equiv1\bmod{m_\mathbf{r}}$. If $2,3\nmid a$, then $a^2\equiv1\bmod{24}$ and hence $a^2\equiv1\bmod{m_\mathbf{r}}$ and $l/a^2\equiv1\bmod{m_\mathbf{r}}$. If $2\mid a$, then $4\mid l$ and hence $2\nmid m_\mathbf{r}$, that is, $m_\mathbf{r}=1$ or $3$. If $m_\mathbf{r}=1$, obviously $l/a^2\equiv1\bmod{m_\mathbf{r}}$; if $m_\mathbf{r}=3$ then $a^2\equiv1\bmod{3}$ and $l\equiv1\bmod{m_\mathbf{r}}$ together imply $l/a^2\equiv1\bmod{m_\mathbf{r}}$. Finally, if $2\nmid a$ and $3\mid a$, then $a^2\equiv1\bmod{8}$ and $3\mid l$. The latter implies $3\nmid m_\mathbf{r}$ and hence $m_\mathbf{r} \mid 8$. Thus, $a^2\equiv1\bmod{m_\mathbf{r}}$, which implies $l/a^2\equiv1\bmod{m_\mathbf{r}}$.

Therefore, we can insert the induction hypothesis into each $c_f\left(\frac{l}{a^2}\right)$ (if there exists at least one) in the right-hand side of \eqref{eq:recursivecf2} and obtain
\begin{multline*}
c_f(l)=-r_1\cdot\left(\sum_{\twoscript{a\mid l}{(a,N)=1}}\legendre{a}{\varPi}a^{k-1}\frac{\psi_{l,\mathbf{r}}(a,1)}{\psi_{l,\mathbf{r}}(1,1)}\cdot\sum_{t\mid(a,l/a)}\frac{\mu(t)}{t}\right.\\
\left.-\sum_{\twoscript{a_1^2\mid l,\,a_1>1}{(a_1,N)=1}}\sum_{\twoscript{a_2\mid l/a_1^2}{(a_2,N)=1}}\legendre{a_1a_2}{\varPi}(a_1a_2)^{k-1}\frac{\psi_{l,\mathbf{r}}(a_1,1)\psi_{l/a_1^2,\mathbf{r}}(a_2,1)}{\psi_{l,\mathbf{r}}(1,1)\psi_{l/a_1^2,\mathbf{r}}(1,1)}\cdot\frac{\mu(a_1)}{a_1}\right).
\end{multline*}
One can verify immediately using the definition that
\begin{equation}
\label{eq:psiIdentity}
\frac{\psi_{l,\mathbf{r}}(a_1,1)\psi_{l/a_1^2,\mathbf{r}}(a_2,1)}{\psi_{l,\mathbf{r}}(1,1)\psi_{l/a_1^2,\mathbf{r}}(1,1)}=\frac{\psi_{l,\mathbf{r}}(a_1a_2,1)}{\psi_{l,\mathbf{r}}(1,1)}
\end{equation}
if $2\nmid l$, or if $2\mid l$, $2\mid N$, or if $l\equiv2\bmod{4}$, $2\nmid N$. However, if $4\mid l$ and $2\nmid N$, then there is no reason that \eqref{eq:psiIdentity} must hold for all $a_1,a_2$ in general. Nevertheless, if we have $2\mid k+\delta$, then \eqref{eq:psiIdentity} holds with both sides equal to $1$, which fortunately is the case for all entries with $2\nmid N$ in Table \ref{table:etaQuotients}. It follows that
\begin{align*}
c_f(l)&=-r_1\cdot\sum_{\twoscript{a\mid l}{(a,N)=1}}\legendre{a}{\varPi}a^{k-1}\frac{\psi_{l,\mathbf{r}}(a,1)}{\psi_{l,\mathbf{r}}(1,1)}\cdot\left(\sum_{t\mid(a,l/a)}\frac{\mu(t)}{t}-\sum_{\twoscript{a_1a_2=a}{a_1^2a_2\mid l,\, a_1>1}}\frac{\mu(a_1)}{a_1}\right)\\
&=-r_1\cdot\sum_{\twoscript{a\mid l}{(a,N)=1}}\legendre{a}{\varPi}a^{k-1}\frac{\psi_{l,\mathbf{r}}(a,1)}{\psi_{l,\mathbf{r}}(1,1)}
\end{align*}
which concludes the induction step, hence the whole proof.
\end{proof}

\section{The proof of Main Theorem, Part II}
\label{sec:proof2}
\begin{lemm}
\label{lemm:24Nchar}
Let $N$ be a positive integer and for each positive divisor $n$ of $N$ let $r_n$ be an integer. Set $\mathbf{r}:=(r_n)_{n\mid N}$ and $f_\mathbf{r}=\prod_{n\mid N}\eta(n\tau)^{r_n}$. Let $k$ and $\varPi$ be as in \eqref{eq:kvarPimr}. Let $\varPi'$ be the square-free part of $\varPi$. Suppose that $k\in\numZ$, $\sum_{n\mid N}nr_n\equiv0\bmod{24}$, and that $f_\mathbf{r}$ has no poles at all cusps of $\Gamma_0(N)\backslash\uhp$, then $f_\mathbf{r}\in M_k\left(\Gamma_0(24N),\legendre{D_{\mathbf{r}}}{\cdot}\right)$ where
\begin{equation*}
D_{\mathbf{r}}:=\begin{dcases}
(-1)^k\varPi' & \text{if }(-1)^k\varPi'\equiv1\bmod{4}\\
(-1)^k4\cdot\varPi' & \text{if }(-1)^k\varPi'\equiv2,3\bmod{4}.
\end{dcases}
\end{equation*}
\end{lemm}
This lemma is closely related to \cite[Theorem 1.64]{Ono04}. We present a short proof here for the reader's convenience.
\begin{proof}
By the assumptions we have $f_\mathbf{r}\in M_k\left(\Gamma_0(N),\chi_{\mathbf{r}}\right)$ where $\chi_{\mathbf{r}}\colon\Gamma_0(N)\to\numC^\times$ is the multiplier system of $f_\mathbf{r}$. Using\footnote{See also \cite[Theorem 5.8.1]{CS17}. However, there the formula for $v(\gamma)$ has a minor error when $c=0$ and $d=-1$.} Petersson's formula \cite[eq. (15)]{ZZ23} we have, for $\tbtmat{a}{b}{c}{d}\in\Gamma_0(24N)$, that
\begin{align*}
\chi_{\mathbf{r}}\tbtMat{a}{b}{c}{d}&=\prod_{n\mid N}v_\eta^{r_n}\tbtMat{a}{bn}{c/n}{d}\\
&=\prod_{n\mid N}\left(\legendre{c/n}{d}\etp{\frac{1}{24}\left((a-2d)c/n-bnd((c/n)^2-1)+3d-3\right)}\right)^{r_n}\\
&=\legendre{(-1)^k\varPi}{d}=\legendre{D_{\mathbf{r}}}{d}.
\end{align*}
In the third step above we have used the assumptions $24\mid\sum_{n\mid N}nr_n$ and $24\mid c/n$. This implies $f_\mathbf{r}\in M_k\left(\Gamma_0(24N),\legendre{D_{\mathbf{r}}}{\cdot}\right)$.
\end{proof}
Since the functions $f_\mathbf{r}$ corresponding to the entries $(N, \mathbf{r})$ in Table \ref{table:etaQuotients} are integral-weight ($k=1,2,3$) holomorphic eta-quotients nonvanishing at infinity, we have $f_\mathbf{r}\in M_k\left(\Gamma_0(24N),\legendre{D_{\mathbf{r}}}{\cdot}\right)$ by the above lemma.

\begin{proof}[Proof of Main Theorem, Part II]
Let $(N, \mathbf{r})$ be a pair in Table \ref{table:etaQuotients}, then $f_\mathbf{r}\in M_k\left(\Gamma_0(24N),\legendre{D_{\mathbf{r}}}{\cdot}\right)$ as we have just seen. According to \cite[Theorems 4.5.2, 4.6.2, and 4.8.1]{DS05}, the Eisenstein subspace of $M_k\left(\Gamma_0(24N),\legendre{D_{\mathbf{r}}}{\cdot}\right)$ is generated by a linearly independent subset of $E_k^{\psi,\phi,t}$ where $\psi$ and $\phi$ are primitive Dirichlet characters of conductors $u$ and $v$ respectively with $\psi\phi(-1)=(-1)^k$ and where $t\in\numgeq{Z}{1}$ such that $uvt\mid24N$. (When $k=1$ we have $E_1^{\psi,\phi,t}$=$E_1^{\phi,\psi,t}$, so only one of these two is required.) We now aim to solve the linear equation $f_\mathbf{r}=\sum_{\psi,\phi,t}c_{\psi,\phi,t}E_k^{\psi,\phi,t}$ for $c_{\psi,\phi,t}\in\numC$. By the Sturm bound (c.f., e.g., \cite[Coro. 5.6.14]{CS17}), we need only to solve this equation with each function $f_\mathbf{r}$ and $E_k^{\psi,\phi,t}$ replaced by its $q$-polynomial up to the $q^{[mk/12]}$-term where $m$ is defined by \eqref{eq:mofGamma0N} with $N$ replaced by $24N$. In this way we have solved this equation for all $(N, \mathbf{r})$ in Table \ref{table:etaQuotients} and it is fortunate that all $\psi,\phi$ involved are quadratic characters that can be represented by $\legendre{D}{\cdot}$ with $D$ a fundamental discriminant. All solutions are presented after Table \ref{table:etaQuotients} in Appendix \ref{sec:The table}.

For instance, to prove \eqref{eq:exampetaEis}, note that we have now $N=24$, $m=2^7\cdot3^2$ and $[mk/12]=96$. Hence we shall compare the Fourier expansions of both sides up to $q^{96}$-term. Since both expansions are
\begin{multline*}
1 + q - q^4 + 2q^6 + 2q^7 - 2q^9 - 2q^{10} - 4q^{15} - q^{16} - 2q^{22} + 2q^{24} + 3q^{25} - 2q^{28} + 2q^{31} - 4q^{33} + 2q^{36} - 2q^{40} + 4q^{42} + 3q^{49} \\+ 2q^{54} + 4q^{55} - 2q^{58} + 4q^{60} - 4q^{63} - q^{64} - 4q^{70} + 2q^{73} + 2q^{79} - 2q^{81} - 4q^{87} - 2q^{88} + 4q^{90} + 2q^{96} +O(q^{97})
\end{multline*}
we conclude that \eqref{eq:exampetaEis} holds.
\end{proof}

\section{Some remarks}
\label{sec:remarks}

\textbf{Incompleteness of the list.} Table \ref{table:etaQuotients} is far from complete. This is because even if $\dim_\numC M_k(\Gamma_0(N), \chi_\mathbf{r})\geq2$, $f_\mathbf{r}$ can still be a Hecke eigenform. Another reason is that there may exist $f_\mathbf{r}$ such that $\dim_\numC M_k(\Gamma_0(N), \chi_\mathbf{r})=1$ but Theorem \ref{thm:dimension} is not applicable. Thus, we pose an open problem here: find out all holomorphic eta-quotients $\prod_{n\mid N}\eta(n\tau)^{r_n}$ ($r_n\in\numZ$) such that $\sum_{n\mid N}r_n$ is even, $\sum_{n\mid N}n\cdot r_n=0$ and \eqref{eq:main} holds for $l\equiv1\bmod{m_{\mathbf{r}}}$. The eta-quotients in Martin's list \cite[Table I]{Mar96} with the property $\sum_{n\mid N}n\cdot r_n=0$ are of this kind.

\textbf{The case of half-integral weight.} What can we say about $c_f(l)$ if $\prod_{n\mid N}\eta(n\tau)^{r_n}$ is of half-integral weight? An open problem: try to deduce a uniform closed formula for $c_f(l)$. The first step should be to simplify \eqref{eq:cTlfbycf} by working out the Gauss sum over $b$. This has been done in a special case; cf. \cite[Proposition 9.31]{Zhu24}. Some explicit formulas for Fourier coefficients of weight $3/2$ eta-quotients, among others, were given in \cite{LO13}.

\textbf{Are the functions in Table \ref{table:etaQuotients} common Hecke eigenfunctions?} Let $f_\mathbf{r}$ be the holomorphic eta-quotient corresponding to an entry $(N, \mathbf{r})$ in Table \ref{table:etaQuotients}. By Lemma \ref{lemm:24Nchar} we have $f_\mathbf{r}\in M_k\left(\Gamma_0(N),\chi_{\mathbf{r}}\right)\subseteq M_k\left(\Gamma_0(24N),\legendre{D_{\mathbf{r}}}{\cdot}\right)$. A key fact is that the Hecke operators on $M_k\left(\Gamma_0(N),\chi_{\mathbf{r}}\right)$ and those on $M_k\left(\Gamma_0(24N),\legendre{D_{\mathbf{r}}}{\cdot}\right)$ are totally different. The theory of the operators on $M_k\left(\Gamma_0(24N),\legendre{D_{\mathbf{r}}}{\cdot}\right)$ was first systematically studied by Atkin and Lehner \cite{AL70} and is now a standard tool in the area of classical modular forms. On the other hand, the operators on $M_k\left(\Gamma_0(N),\chi_{\mathbf{r}}\right)$ is less well-known when $\chi_{\mathbf{r}}$ is not induced by Dirichlet characters; c.f. \cite[Prop. 3.2]{ZZ23} and \cite[\S 9.3]{Zhu24}. Since we have using $T_l$ to denote the operators on $M_k\left(\Gamma_0(N),\chi_{\mathbf{r}}\right)$ (see Theorem \ref{thm:cTlfbycf}), let us use $\mathcal{T}_l$ to denote the operators on $M_k\left(\Gamma_0(24N),\legendre{D_{\mathbf{r}}}{\cdot}\right)$; c.f. \cite[Prop 10.2.5(a)]{CS17}. Then there are no reason that the eigenfunctions of $T_l$ and those of $\mathcal{T}_l$ should be the same. This explains the ``paradox'' posed in Section \ref{sec:Introduction}. Let us show this more concretely with the example \eqref{eq:exampetaEis}. By \eqref{eq:Tlfclf} we know that $\eta(\tau)^{-1}\eta(2\tau)^2\eta(3\tau)\eta(6\tau)^{-3}\eta(12\tau)^5\eta(24\tau)^{-2}$ is an eigenfunction of $T_l$ with $l\equiv1\bmod{6}$, which are all possible double coset oeprators on $M_1\left(\Gamma_0(24),\chi_{\mathbf{r}}\right)$. On the other hand, each of $E_1^{-3,8,1}$, $E_1^{-3,8,2}$, $E_1^{-3,8,3}$, $E_1^{-3,8,6}$, $E_1^{-24,1,1}$, $E_1^{-24,1,2}$, $E_1^{-24,1,3}$ and $E_1^{-24,1,6}$ is an eigenfunction of $\mathcal{T}_p$, where $p\neq2,3$ is a prime; c.f. \cite[Prop. 5.2.3]{DS05}. Since their eigenvalues are not necessarily the same, $\eta(\tau)^{-1}\eta(2\tau)^2\eta(3\tau)\eta(6\tau)^{-3}\eta(12\tau)^5\eta(24\tau)^{-2}$ may not be an eigenfunction of $\mathcal{T}_p$. Conversely, $E_1^{-3,8,1}$ (or the other seven) may not be an eigenfunction of $T_l$.

\textbf{The dimension formulas.} Theorem \ref{thm:dimension} is a dimension formula for $M_k(\Gamma_0(N), \chi)$ where $\chi$ is the multiplier system of an eta-quotient $\prod_{n\mid N}\eta(n\tau)^{r_n}$. The dimension formula for $\chi$ being a Dirichlet character and $k\in\numZ$ is well-known; c.f. \cite[Theorem 7.4.1]{CS17}. When both formulas apply, they give the same result. Here we illustrate this by an example. We consider $f=\eta(\tau)^{-6}\eta(2\tau)^{17}\eta(4\tau)^{-7}\in M_{2}(\Gamma_0(96),\legendre{8}{\cdot})$ (see Lemma \ref{lemm:24Nchar}). First we calculate $\dim_\numC M_{2}(\Gamma_0(96),\legendre{8}{\cdot})$ by \cite[Theorem 7.4.1]{CS17}. We have
\begin{align*}
A_1&=\frac{k-1}{12}N\prod_{p\mid N}\left(1+\frac{1}{p}\right)=16,\\
A_2&=\left(\frac{k-1}{3}-\left[\frac{k}{3}\right]\right)\sum_{\twoscript{x\bmod N}{x^2+x+1\equiv0\bmod{N}}}\legendre{8}{x}+\left(\frac{k-1}{4}-\left[\frac{k}{4}\right]\right)\sum_{\twoscript{x\bmod N}{x^2+1\equiv0\bmod{N}}}\legendre{8}{x}=0\\
A_3&=\frac{1}{2}\sum_{\twoscript{0<d\mid N}{(d,N/d)\mid N/8}}\phi(d,N/d)=8,
\end{align*}
and hence $\dim_\numC M_{2}(\Gamma_0(96),\legendre{8}{\cdot})=A_1-A_2+A_3=24$. Then we calculate $\dim_\numC M_{2}(\Gamma_0(96),\legendre{8}{\cdot})$ by Theorem \ref{thm:dimension}. We have
\begin{equation*}
\begin{pmatrix}
x_1\\
x_2\\
x_3\\
x_4\\
x_6\\
x_8\\
x_{12}\\
x_{16}\\
x_{24}\\
x_{32}\\
x_{48}\\
x_{96}
\end{pmatrix}=
\begin{pmatrix}
96 &48 &32 &24 &16 &12 &8 &6 &4 &3 &2 &1 \\
24 &48 &8 &24 &16 &12 &8 &6 &4 &3 &2 &1 \\
32 &16 &96 &8 &48 &4 &24 &2 &12 &1 &6 &3 \\
6 &12 &2 &24 &4 &12 &8 &6 &4 &3 &2 &1 \\
8 &16 &24 &8 &48 &4 &24 &2 &12 &1 &6 &3 \\
3 &6 &1 &12 &2 &24 &4 &12 &8 &6 &4 &2 \\
2 &4 &6 &8 &12 &4 &24 &2 &12 &1 &6 &3 \\
3 &6 &1 &12 &2 &24 &4 &48 &8 &24 &16 &8 \\
1 &2 &3 &4 &6 &8 &12 &4 &24 &2 &12 &6 \\
3 &6 &1 &12 &2 &24 &4 &48 &8 &96 &16 &32 \\
1 &2 &3 &4 &6 &8 &12 &16 &24 &8 &48 &24 \\
1 &2 &3 &4 &6 &8 &12 &16 &24 &32 &48 &96
\end{pmatrix}\cdot\begin{pmatrix}
r_1\\
r_2\\
r_3\\
r_4\\
r_6\\
r_8\\
r_{12}\\
r_{16}\\
r_{24}\\
r_{32}\\
r_{48}\\
r_{96}
\end{pmatrix}=\begin{pmatrix}
72\\
504\\
24\\
0\\
168\\
0\\
0\\
0\\
0\\
0\\
0\\
0
\end{pmatrix}.
\end{equation*}
Inserting this into \eqref{eq:dimension} gives $\dim_\numC M_{2}(\Gamma_0(96),\legendre{8}{\cdot})=16+0+0+8=24$, which coincides with the value given by the classical formula. Also note that if we take the point of view $\eta(\tau)^{-6}\eta(2\tau)^{17}\eta(4\tau)^{-7}\in M_{2}(\Gamma_0(4),\chi)$, then $\chi$ can not be induced by a Dirichlet character modulo $4$. Therefore, \cite[Theorem 7.4.1]{CS17} is not applicable to $\dim_\numC M_{2}(\Gamma_0(4),\chi)$. Indeed, it does follow from the fact $\dim_\numC M_{2}(\Gamma_0(4),\chi)=1$, calculated using Theorem \ref{thm:dimension}, that \eqref{eq:Tlfclf} holds for $f_{\mathbf{r}}=\eta(\tau)^{-6}\eta(2\tau)^{17}\eta(4\tau)^{-7}$.

\textbf{Relations between the two parts of the main theorem.} By the second part of the main theorem, we actually have a formula for all coefficients of $f_{\mathbf{r}}$ since the coefficients of $E_k^{D_1,D_2,t}$ can be seen as explicit. However, the first part also has its independent value since the formula \eqref{eq:main} takes a very simple form and is not a direct consequence of the Eisenstein series decomposition given by the second part, except for $f_{\mathbf{r}}=\eta(\tau)^{-4}\eta(2\tau)^{10}\eta(4\tau)^{-4}$ and $\eta(\tau)^{-2}\eta(2\tau)^3\eta(4\tau)^3\eta(8\tau)^{-2}$.

\textbf{The SageMath code.} The code used to generate Table \ref{table:etaQuotients} is available from the author on reasonable e-mail request. Only by looking at this code can one verify that Table \ref{table:etaQuotients} truly exhausts all admissible eta-quotients of type I (cf. \cite[Section 9]{Zhu24} for the definition) with $x_N=0$. There is as well code used to generate all identities in the second part of the main theorem and is also available from reasonable e-mail requests.

\begin{appendix}
\section{The table}
\label{sec:The table}
The contents of Table \ref{table:etaQuotients} have been explained in the first part of the main theorem and the paragraph following it. We explain the usage here. The item $1^{2}2^{-1}5^{2}10^{-1}$, for instance, means $\mathbf{r}=(r_n)_{n\mid10}$ with $r_1=2$, $r_2=-1$, $r_5=2$ and $r_{10}=-1$. Hence for this choice, $f_\mathbf{r}$ in the main theorem is the eta-quotient
\begin{equation*}
\eta(\tau)^2\eta(2\tau)^{-1}\eta(5\tau)^2\eta(10\tau)^{-1}=\prod_{m=1}^{+\infty}(1-q^m)^2(1-q^{2m})^{-1}(1-q^{5m})^{2}(1-q^{10m})^{-1}.
\end{equation*}
The suffix ``p'' in $10$p means $f_\mathbf{r}$ is a primitive eta-quotient, namely, there exist no eta-quotient $g_{\mathbf{r}'}$ and integer $l\in\numgeq{Z}{2}$ such that $f_\mathbf{r}(\tau)=g_{\mathbf{r}'}(l\tau)$. The first part of our main theorem now says that the coefficient of the $q^l$-term is
\begin{equation*}
-2\cdot\sum_{\twoscript{a\mid l}{(a,10)=1}}\legendre{a}{5}\cdot(-1)^{\frac{a-1}{2}}
\end{equation*}
if $l\equiv1\bmod{4}$. (The information $m_\mathbf{r}$, $k$ and $\delta$ are also listed in the table.) Note that $\varPi'$ in Table \ref{table:etaQuotients} means the square-free part of $\varPi$. Since $\legendre{a}{\varPi}=\legendre{a}{\varPi'}$, it is more convenient to use $\varPi'$.

\begin{longtable}{llllll|llllll}
\caption{Eta-quotients encountered in the main theorem\label{table:etaQuotients}}\\
\toprule
$N$ & $\mathbf{r}$ & $m_\mathbf{r}$ & $k$ & $\delta$ &$\varPi'$ & $N$ & $\mathbf{r}$ & $m_\mathbf{r}$ & $k$ & $\delta$ &$\varPi'$\\
\midrule
\endfirsthead
$N$ & $\mathbf{r}$ & $m_\mathbf{r}$ & $k$ & $\delta$ &$\varPi'$ & $N$ & $\mathbf{r}$ & $m_\mathbf{r}$ & $k$ & $\delta$ &$\varPi'$\\
\midrule
\endhead
$2$p & $1^{4}2^{-2}$ & $4$ & $1$ & $0$ & $1$ & $2$p & $1^{8}2^{-4}$ & $2$ & $2$ & $0$ & $1$\\
$2$p & $1^{12}2^{-6}$ & $4$ & $3$ & $0$ & $1$ & $3$p & $1^{3}3^{-1}$ & $3$ & $1$ & $1$ & $3$\\
$3$p & $1^{6}3^{-2}$ & $3$ & $2$ & $0$ & $1$ & $4$p & $1^{-4}2^{10}4^{-4}$ & $1$ & $1$ & $0$ & $1$\\
$4$p & $1^{-2}2^{7}4^{-3}$ & $8$ & $1$ & $0$ & $2$ & $4$ & $1^{0}2^{4}4^{-2}$ & $4$ & $1$ & $0$ & $1$\\
$4$p & $1^{2}2^{1}4^{-1}$ & $8$ & $1$ & $0$ & $2$ & $4$p & $1^{-6}2^{17}4^{-7}$ & $8$ & $2$ & $0$ & $2$\\
$4$p & $1^{-4}2^{14}4^{-6}$ & $4$ & $2$ & $0$ & $1$ & $4$p & $1^{-2}2^{11}4^{-5}$ & $8$ & $2$ & $0$ & $2$\\
$4$ & $1^{0}2^{8}4^{-4}$ & $2$ & $2$ & $0$ & $1$ & $4$p & $1^{2}2^{5}4^{-3}$ & $8$ & $2$ & $0$ & $2$\\
$4$p & $1^{4}2^{2}4^{-2}$ & $4$ & $2$ & $0$ & $1$ & $4$p & $1^{6}2^{-1}4^{-1}$ & $8$ & $2$ & $0$ & $2$\\
$4$p & $1^{-2}2^{15}4^{-7}$ & $8$ & $3$ & $0$ & $2$ & $4$ & $1^{0}2^{12}4^{-6}$ & $4$ & $3$ & $0$ & $1$\\
$4$p & $1^{2}2^{9}4^{-5}$ & $8$ & $3$ & $0$ & $2$ & $6$p & $1^{-2}2^{4}3^{2}6^{-2}$ & $12$ & $1$ & $0$ & $1$\\
$6$p & $1^{-1}2^{2}3^{3}6^{-2}$ & $6$ & $1$ & $1$ & $3$ & $6$ & $1^{0}2^{0}3^{4}6^{-2}$ & $4$ & $1$ & $0$ & $1$\\
$6$ & $1^{0}2^{3}3^{0}6^{-1}$ & $3$ & $1$ & $1$ & $3$ & $6$p & $1^{1}2^{1}3^{1}6^{-1}$ & $12$ & $1$ & $0$ & $1$\\
$6$p & $1^{2}2^{-1}3^{2}6^{-1}$ & $2$ & $1$ & $1$ & $3$ & $6$p & $1^{0}2^{3}3^{4}6^{-3}$ & $12$ & $2$ & $1$ & $3$\\
$6$p & $1^{2}2^{2}3^{2}6^{-2}$ & $6$ & $2$ & $0$ & $1$ & $6$p & $1^{3}2^{0}3^{3}6^{-2}$ & $12$ & $2$ & $1$ & $3$\\
$8$p & $1^{-2}2^{3}4^{3}8^{-2}$ & $1$ & $1$ & $0$ & $2$ & $8$p & $1^{-2}2^{5}4^{0}8^{-1}$ & $8$ & $1$ & $0$ & $1$\\
$8$ & $1^{0}2^{-2}4^{7}8^{-3}$ & $8$ & $1$ & $0$ & $2$ & $8$ & $1^{0}2^{0}4^{4}8^{-2}$ & $4$ & $1$ & $0$ & $1$\\
$8$ & $1^{0}2^{2}4^{1}8^{-1}$ & $8$ & $1$ & $0$ & $2$ & $8$p & $1^{2}2^{-3}4^{5}8^{-2}$ & $2$ & $1$ & $0$ & $2$\\
$8$p & $1^{2}2^{-1}4^{2}8^{-1}$ & $8$ & $1$ & $0$ & $1$ & $8$p & $1^{-2}2^{5}4^{4}8^{-3}$ & $8$ & $2$ & $0$ & $1$\\
$8$ & $1^{0}2^{2}4^{5}8^{-3}$ & $8$ & $2$ & $0$ & $2$ & $8$ & $1^{0}2^{4}4^{2}8^{-2}$ & $4$ & $2$ & $0$ & $1$\\
$8$ & $1^{0}2^{6}4^{-1}8^{-1}$ & $8$ & $2$ & $0$ & $2$ & $8$p & $1^{2}2^{-1}4^{6}8^{-3}$ & $8$ & $2$ & $0$ & $1$\\
$9$ & $1^{0}3^{3}9^{-1}$ & $3$ & $1$ & $1$ & $3$ & $9$ & $1^{0}3^{6}9^{-2}$ & $3$ & $2$ & $0$ & $1$\\
$10$p & $1^{2}2^{-1}5^{2}10^{-1}$ & $4$ & $1$ & $0$ & $5$ & $12$p & $1^{-2}2^{4}3^{0}4^{0}6^{1}12^{-1}$ & $24$ & $1$ & $0$ & $2$\\
$12$p & $1^{-2}2^{5}3^{0}4^{-2}6^{2}12^{-1}$ & $8$ & $1$ & $1$ & $6$ & $12$p & $1^{-2}2^{5}3^{2}4^{-2}6^{-1}12^{0}$ & $4$ & $1$ & $1$ & $3$\\
$12$p & $1^{-1}2^{2}3^{-1}4^{0}6^{4}12^{-2}$ & $12$ & $1$ & $1$ & $3$ & $12$p & $1^{-1}2^{2}3^{1}4^{0}6^{1}12^{-1}$ & $24$ & $1$ & $1$ & $6$\\
$12$p & $1^{-1}2^{4}3^{-1}4^{-1}6^{2}12^{-1}$ & $3$ & $1$ & $0$ & $1$ & $12$p & $1^{-1}2^{4}3^{1}4^{-1}6^{-1}12^{0}$ & $24$ & $1$ & $0$ & $2$\\
$12$p & $1^{0}2^{-1}3^{-2}4^{2}6^{6}12^{-3}$ & $24$ & $1$ & $1$ & $6$ & $12$ & $1^{0}2^{-1}3^{0}4^{2}6^{3}12^{-2}$ & $6$ & $1$ & $1$ & $3$\\
$12$p & $1^{0}2^{-1}3^{2}4^{2}6^{0}12^{-1}$ & $24$ & $1$ & $1$ & $6$ & $12$ & $1^{0}2^{0}3^{0}4^{0}6^{4}12^{-2}$ & $4$ & $1$ & $0$ & $1$\\
$12$ & $1^{0}2^{1}3^{0}4^{1}6^{1}12^{-1}$ & $12$ & $1$ & $0$ & $1$ & $12$p & $1^{0}2^{2}3^{-2}4^{-1}6^{5}12^{-2}$ & $8$ & $1$ & $1$ & $6$\\
$12$ & $1^{0}2^{2}3^{0}4^{-1}6^{2}12^{-1}$ & $2$ & $1$ & $1$ & $3$ & $12$p & $1^{0}2^{2}3^{2}4^{-1}6^{-1}12^{0}$ & $8$ & $1$ & $1$ & $6$\\
$12$p & $1^{1}2^{-1}3^{-1}4^{1}6^{4}12^{-2}$ & $24$ & $1$ & $1$ & $6$ & $12$p & $1^{1}2^{-1}3^{1}4^{1}6^{1}12^{-1}$ & $12$ & $1$ & $1$ & $3$\\
$12$p & $1^{1}2^{1}3^{-1}4^{0}6^{2}12^{-1}$ & $24$ & $1$ & $0$ & $2$ & $12$p & $1^{2}2^{-2}3^{-2}4^{2}6^{4}12^{-2}$ & $3$ & $1$ & $0$ & $1$\\
$12$p & $1^{2}2^{-2}3^{0}4^{2}6^{1}12^{-1}$ & $24$ & $1$ & $0$ & $2$ & $12$p & $1^{2}2^{-1}3^{-2}4^{0}6^{5}12^{-2}$ & $4$ & $1$ & $1$ & $3$\\
$12$p & $1^{2}2^{-1}3^{0}4^{0}6^{2}12^{-1}$ & $8$ & $1$ & $1$ & $6$ & $12$ & $1^{0}2^{3}3^{0}4^{0}6^{3}12^{-2}$ & $12$ & $2$ & $1$ & $3$\\
$14$ & $1^{0}2^{0}7^{4}14^{-2}$ & $4$ & $1$ & $0$ & $1$ & $15$ & $1^{0}3^{0}5^{3}15^{-1}$ & $3$ & $1$ & $1$ & $3$\\
$16$ & $1^{0}2^{-2}4^{5}8^{0}16^{-1}$ & $8$ & $1$ & $0$ & $1$ & $16$ & $1^{0}2^{0}4^{2}8^{1}16^{-1}$ & $8$ & $1$ & $0$ & $2$\\
$16$ & $1^{0}2^{2}4^{-1}8^{2}16^{-1}$ & $8$ & $1$ & $0$ & $1$ & $16$ & $1^{0}2^{0}4^{6}8^{-1}16^{-1}$ & $8$ & $2$ & $0$ & $2$\\
$18$ & $1^{0}2^{0}3^{-2}6^{4}9^{2}18^{-2}$ & $12$ & $1$ & $0$ & $1$ & $18$ & $1^{0}2^{0}3^{0}6^{3}9^{0}18^{-1}$ & $3$ & $1$ & $1$ & $3$\\
$18$ & $1^{0}2^{0}3^{1}6^{1}9^{1}18^{-1}$ & $12$ & $1$ & $0$ & $1$ & $20$p & $1^{-2}2^{5}4^{-2}5^{0}10^{2}20^{-1}$ & $8$ & $1$ & $0$ & $10$\\
$20$ & $1^{0}2^{0}4^{0}5^{-2}10^{7}20^{-3}$ & $8$ & $1$ & $0$ & $2$ & $20$ & $1^{0}2^{0}4^{0}5^{2}10^{1}20^{-1}$ & $8$ & $1$ & $0$ & $2$\\
$20$p & $1^{0}2^{2}4^{-1}5^{-2}10^{5}20^{-2}$ & $8$ & $1$ & $0$ & $10$ & $20$p & $1^{0}2^{2}4^{-1}5^{2}10^{-1}20^{0}$ & $8$ & $1$ & $0$ & $10$\\
$20$p & $1^{2}2^{-1}4^{0}5^{0}10^{2}20^{-1}$ & $8$ & $1$ & $0$ & $10$ & $24$p & $1^{-2}2^{5}3^{0}4^{-2}6^{0}8^{0}12^{2}24^{-1}$ & $8$ & $1$ & $1$ & $3$\\
$24$p & $1^{-2}2^{6}3^{0}4^{-3}6^{-1}8^{1}12^{2}24^{-1}$ & $3$ & $1$ & $0$ & $2$ & $24$p & $1^{-1}2^{0}3^{1}4^{5}6^{-1}8^{-2}12^{0}24^{0}$ & $6$ & $1$ & $0$ & $2$\\
$24$p & $1^{-1}2^{2}3^{1}4^{-1}6^{-1}8^{2}12^{1}24^{-1}$ & $24$ & $1$ & $0$ & $1$ & $24$p & $1^{-1}2^{2}3^{1}4^{0}6^{-3}8^{0}12^{5}24^{-2}$ & $6$ & $1$ & $1$ & $6$\\
$24$ & $1^{0}2^{-2}3^{0}4^{5}6^{2}8^{-2}12^{-1}24^{0}$ & $4$ & $1$ & $1$ & $3$ & $24$ & $1^{0}2^{-1}3^{0}4^{4}6^{1}8^{-1}12^{-1}24^{0}$ & $24$ & $1$ & $0$ & $2$\\
$24$ & $1^{0}2^{0}3^{-2}4^{0}6^{5}8^{0}12^{0}24^{-1}$ & $8$ & $1$ & $0$ & $1$ & $24$p & $1^{0}2^{0}3^{-2}4^{2}6^{5}8^{-1}12^{-2}24^{0}$ & $8$ & $1$ & $1$ & $3$\\
$24$ & $1^{0}2^{0}3^{0}4^{-1}6^{2}8^{2}12^{0}24^{-1}$ & $24$ & $1$ & $1$ & $6$ & $24$ & $1^{0}2^{0}3^{0}4^{0}6^{0}8^{0}12^{4}24^{-2}$ & $4$ & $1$ & $0$ & $1$\\
$24$ & $1^{0}2^{0}3^{0}4^{1}6^{0}8^{1}12^{1}24^{-1}$ & $12$ & $1$ & $0$ & $1$ & $24$ & $1^{0}2^{0}3^{2}4^{0}6^{-1}8^{0}12^{2}24^{-1}$ & $8$ & $1$ & $0$ & $1$\\
$24$p & $1^{0}2^{0}3^{2}4^{2}6^{-1}8^{-1}12^{0}24^{0}$ & $8$ & $1$ & $1$ & $3$ & $24$p & $1^{0}2^{1}3^{-2}4^{-1}6^{4}8^{1}12^{0}24^{-1}$ & $3$ & $1$ & $1$ & $6$\\
$24$ & $1^{0}2^{1}3^{0}4^{-1}6^{-1}8^{1}12^{4}24^{-2}$ & $24$ & $1$ & $1$ & $6$ & $24$p & $1^{0}2^{1}3^{2}4^{-1}6^{-2}8^{1}12^{2}24^{-1}$ & $6$ & $1$ & $1$ & $6$\\
$24$ & $1^{0}2^{2}3^{0}4^{-2}6^{0}8^{2}12^{1}24^{-1}$ & $24$ & $1$ & $0$ & $2$ & $24$ & $1^{0}2^{2}3^{0}4^{-1}6^{-2}8^{0}12^{5}24^{-2}$ & $4$ & $1$ & $1$ & $3$\\
$24$p & $1^{1}2^{-3}3^{-1}4^{6}6^{2}8^{-2}12^{-1}24^{0}$ & $3$ & $1$ & $0$ & $2$ & $24$p & $1^{1}2^{-1}3^{-1}4^{0}6^{2}8^{2}12^{0}24^{-1}$ & $24$ & $1$ & $0$ & $1$\\
$24$p & $1^{1}2^{-1}3^{-1}4^{1}6^{0}8^{0}12^{4}24^{-2}$ & $3$ & $1$ & $1$ & $6$ & $24$p & $1^{2}2^{-1}3^{0}4^{0}6^{0}8^{0}12^{2}24^{-1}$ & $8$ & $1$ & $1$ & $3$\\
$24$p & $1^{2}2^{0}3^{0}4^{-1}6^{-1}8^{1}12^{2}24^{-1}$ & $6$ & $1$ & $0$ & $2$ & $32$ & $1^{0}2^{0}4^{0}8^{2}16^{1}32^{-1}$ & $8$ & $1$ & $0$ & $2$\\
$32$ & $1^{0}2^{0}4^{2}8^{-1}16^{2}32^{-1}$ & $8$ & $1$ & $0$ & $1$ & $36$ & $1^{0}2^{0}3^{-2}4^{0}6^{4}9^{0}12^{0}18^{1}36^{-1}$ & $24$ & $1$ & $0$ & $2$\\
$36$ & $1^{0}2^{0}3^{-1}4^{0}6^{4}9^{-1}12^{-1}18^{2}36^{-1}$ & $3$ & $1$ & $0$ & $1$ & $36$ & $1^{0}2^{0}3^{-1}4^{0}6^{4}9^{1}12^{-1}18^{-1}36^{0}$ & $24$ & $1$ & $0$ & $2$\\
$36$ & $1^{0}2^{0}3^{1}4^{0}6^{1}9^{-1}12^{0}18^{2}36^{-1}$ & $24$ & $1$ & $0$ & $2$ & $36$ & $1^{0}2^{0}3^{2}4^{0}6^{-2}9^{-2}12^{2}18^{4}36^{-2}$ & $3$ & $1$ & $0$ & $1$\\
$36$ & $1^{0}2^{0}3^{2}4^{0}6^{-2}9^{0}12^{2}18^{1}36^{-1}$ & $24$ & $1$ & $0$ & $2$ & &&&&\\
\bottomrule
\end{longtable}

Below is a list of the identities encountered in the second part of the main theorem, which is presented in a compact form. Since the expansion of a nonprimitive eta-quotient is a direct consequence of the expansion of the corresponding primitive eta-quotient, we only present the identities for 66 primitive eta-quotients in Table \ref{table:etaQuotients}.
\begin{align*}
1^{4}2^{-2}&=[-4, 1, 1]^{-2}[-4, 1, 2]^{4}\\
1^{8}2^{-4}&=[1, 1, 2]^{-12}[1, 1, 4]^{8}\\
1^{12}2^{-6}&=[1, -4, 1]^{2}[1, -4, 2]^{-4}[-4, 1, 1]^{-8}[-4, 1, 2]^{64}\\
1^{3}3^{-1}&=[-3, 1, 1]^{-\frac{3}{2}}[-3, 1, 3]^{\frac{9}{2}}\\
1^{6}3^{-2}&=[1, 1, 3]^{-3}[1, 1, 9]^{\frac{9}{4}}[-3, -3, 1]^{-\frac{9}{4}}\\
1^{-4}2^{10}4^{-4}&=[-4, 1, 1]^{2}\\
1^{-2}2^{7}4^{-3}&=[-4, 8, 1]^{1}[-8, 1, 2]^{-1}[-8, 1, 4]^{2}\\
1^{2}2^{1}4^{-1}&=[-4, 8, 1]^{-1}[-8, 1, 2]^{-1}[-8, 1, 4]^{2}\\
1^{-6}2^{17}4^{-7}&=[1, 8, 2]^{1}[1, 8, 4]^{-2}[-4, -8, 1]^{1}[8, 1, 2]^{4}[8, 1, 4]^{-16}[-8, -4, 1]^{2}\\
1^{-4}2^{14}4^{-6}&=[1, 1, 4]^{1}[1, 1, 8]^{-3}[1, 1, 16]^{2}[-4, -4, 1]^{2}\\
1^{-2}2^{11}4^{-5}&=[1, 8, 2]^{1}[1, 8, 4]^{-2}[-4, -8, 1]^{-1}[8, 1, 2]^{-4}[8, 1, 4]^{16}[-8, -4, 1]^{2}\\
1^{2}2^{5}4^{-3}&=[1, 8, 2]^{1}[1, 8, 4]^{-2}[-4, -8, 1]^{1}[8, 1, 2]^{-4}[8, 1, 4]^{16}[-8, -4, 1]^{-2}\\
1^{4}2^{2}4^{-2}&=[1, 1, 4]^{1}[1, 1, 8]^{-3}[1, 1, 16]^{2}[-4, -4, 1]^{-2}\\
1^{6}2^{-1}4^{-1}&=[1, 8, 2]^{1}[1, 8, 4]^{-2}[-4, -8, 1]^{-1}[8, 1, 2]^{4}[8, 1, 4]^{-16}[-8, -4, 1]^{-2}\\
1^{-2}2^{15}4^{-7}&=[1, -8, 2]^{\frac{1}{3}}[1, -8, 4]^{-\frac{2}{3}}[-4, 8, 1]^{-\frac{1}{3}}[8, -4, 1]^{\frac{4}{3}}[-8, 1, 2]^{-\frac{16}{3}}[-8, 1, 4]^{\frac{128}{3}}\\
1^{2}2^{9}4^{-5}&=[1, -8, 2]^{\frac{1}{3}}[1, -8, 4]^{-\frac{2}{3}}[-4, 8, 1]^{\frac{1}{3}}[8, -4, 1]^{-\frac{4}{3}}[-8, 1, 2]^{-\frac{16}{3}}[-8, 1, 4]^{\frac{128}{3}}\\
1^{-2}2^{4}3^{2}6^{-2}&=[-3, 12, 1]^{\frac{3}{4}}[-3, 12, 2]^{\frac{3}{2}}[-4, 1, 1]^{\frac{1}{4}}[-4, 1, 2]^{-\frac{1}{2}}[-4, 1, 9]^{-\frac{9}{4}}[-4, 1, 18]^{\frac{9}{2}}\\
1^{-1}2^{2}3^{3}6^{-2}&=[-3, 1, 1]^{\frac{1}{2}}[-3, 1, 3]^{-\frac{3}{2}}[-3, 1, 4]^{-2}[-3, 1, 12]^{6}\\
1^{1}2^{1}3^{1}6^{-1}&=[-3, 12, 1]^{-\frac{3}{4}}[-3, 12, 2]^{-\frac{3}{2}}[-4, 1, 1]^{\frac{1}{4}}[-4, 1, 2]^{-\frac{1}{2}}[-4, 1, 9]^{-\frac{9}{4}}[-4, 1, 18]^{\frac{9}{2}}\\
1^{2}2^{-1}3^{2}6^{-1}&=[-3, 1, 1]^{-1}[-3, 1, 4]^{4}\\
1^{0}2^{3}3^{4}6^{-3}&=[1, 12, 1]^{-\frac{1}{4}}[1, 12, 2]^{\frac{1}{2}}[1, 12, 3]^{\frac{3}{4}}[1, 12, 6]^{-\frac{3}{2}}[-3, -4, 1]^{\frac{1}{4}}[-3, -4, 2]^{\frac{1}{2}}\\&[-3, -4, 3]^{\frac{9}{4}}[-3, -4, 6]^{\frac{9}{2}}[-4, -3, 1]^{-\frac{1}{2}}[-4, -3, 2]^{-2}[-4, -3, 3]^{-\frac{3}{2}}[-4, -3, 6]^{-6}\\&[12, 1, 1]^{\frac{1}{2}}[12, 1, 2]^{-2}[12, 1, 3]^{-\frac{9}{2}}[12, 1, 6]^{18}\\
1^{2}2^{2}3^{2}6^{-2}&=[1, 1, 2]^{-\frac{3}{4}}[1, 1, 3]^{-1}[1, 1, 4]^{\frac{1}{2}}[1, 1, 6]^{3}[1, 1, 9]^{\frac{3}{4}}[1, 1, 12]^{-2}\\&[1, 1, 18]^{-\frac{9}{4}}[1, 1, 36]^{\frac{3}{2}}[-3, -3, 1]^{-\frac{3}{4}}[-3, -3, 2]^{-\frac{9}{2}}[-3, -3, 4]^{-6}\\
1^{3}2^{0}3^{3}6^{-2}&=[1, 12, 1]^{-\frac{1}{4}}[1, 12, 2]^{\frac{1}{2}}[1, 12, 3]^{\frac{3}{4}}[1, 12, 6]^{-\frac{3}{2}}[-3, -4, 1]^{-\frac{1}{4}}[-3, -4, 2]^{-\frac{1}{2}}\\&[-3, -4, 3]^{-\frac{9}{4}}[-3, -4, 6]^{-\frac{9}{2}}[-4, -3, 1]^{-\frac{1}{2}}[-4, -3, 2]^{-2}[-4, -3, 3]^{-\frac{3}{2}}[-4, -3, 6]^{-6}\\&[12, 1, 1]^{-\frac{1}{2}}[12, 1, 2]^{2}[12, 1, 3]^{\frac{9}{2}}[12, 1, 6]^{-18}\\
1^{-2}2^{3}4^{3}8^{-2}&=[-8, 1, 1]^{1}\\
1^{-2}2^{5}4^{0}8^{-1}&=[-4, 1, 8]^{-2}[-4, 1, 16]^{4}[-8, 8, 1]^{1}\\
1^{2}2^{-3}4^{5}8^{-2}&=[-8, 1, 1]^{-1}[-8, 1, 2]^{2}\\
1^{2}2^{-1}4^{2}8^{-1}&=[-4, 1, 8]^{-2}[-4, 1, 16]^{4}[-8, 8, 1]^{-1}\\
1^{-2}2^{5}4^{4}8^{-3}&=[1, 1, 16]^{\frac{1}{4}}[1, 1, 32]^{-\frac{3}{4}}[1, 1, 64]^{\frac{1}{2}}[-4, -4, 4]^{-2}[8, 8, 1]^{\frac{1}{2}}[-8, -8, 1]^{\frac{1}{2}}\\
1^{2}2^{-1}4^{6}8^{-3}&=[1, 1, 16]^{\frac{1}{4}}[1, 1, 32]^{-\frac{3}{4}}[1, 1, 64]^{\frac{1}{2}}[-4, -4, 4]^{-2}[8, 8, 1]^{-\frac{1}{2}}[-8, -8, 1]^{-\frac{1}{2}}\\
1^{2}2^{-1}5^{2}10^{-1}&=[-4, 5, 1]^{-\frac{1}{2}}[-4, 5, 2]^{-1}[-20, 1, 1]^{-\frac{1}{2}}[-20, 1, 2]^{1}\\
1^{-2}2^{4}3^{0}4^{0}6^{1}12^{-1}&=[-3, 24, 2]^{\frac{3}{4}}[-3, 24, 4]^{\frac{3}{2}}[-4, 8, 1]^{\frac{1}{4}}[-4, 8, 3]^{\frac{3}{2}}[-4, 8, 9]^{\frac{9}{4}}[-8, 1, 2]^{-\frac{1}{4}}\\&[-8, 1, 4]^{\frac{1}{2}}[-8, 1, 6]^{\frac{3}{2}}[-8, 1, 12]^{-3}[-8, 1, 18]^{-\frac{9}{4}}[-8, 1, 36]^{\frac{9}{2}}[-24, 12, 1]^{\frac{3}{4}}\\
1^{-2}2^{5}3^{0}4^{-2}6^{2}12^{-1}&=[-3, 8, 2]^{\frac{1}{2}}[-3, 8, 4]^{1}[-4, 24, 1]^{\frac{1}{2}}[-8, 12, 1]^{\frac{1}{2}}[-24, 1, 2]^{-\frac{1}{2}}[-24, 1, 4]^{1}\\
1^{-2}2^{5}3^{2}4^{-2}6^{-1}12^{0}&=[-3, 1, 4]^{-1}[-3, 1, 16]^{4}[-4, 12, 1]^{1}\\
1^{-1}2^{2}3^{-1}4^{0}6^{4}12^{-2}&=[-3, 1, 4]^{\frac{1}{2}}[-3, 1, 12]^{-\frac{3}{2}}[-3, 1, 16]^{-2}[-3, 1, 48]^{6}[-4, 12, 1]^{\frac{1}{2}}[-4, 12, 3]^{\frac{3}{2}}\\
1^{-1}2^{2}3^{1}4^{0}6^{1}12^{-1}&=[-3, 8, 2]^{-\frac{1}{4}}[-3, 8, 4]^{-\frac{1}{2}}[-3, 8, 6]^{-\frac{3}{4}}[-3, 8, 12]^{-\frac{3}{2}}[-4, 24, 1]^{\frac{1}{4}}[-4, 24, 3]^{\frac{3}{4}}\\&[-8, 12, 1]^{\frac{1}{4}}[-8, 12, 3]^{-\frac{3}{4}}[-24, 1, 2]^{\frac{1}{4}}[-24, 1, 4]^{-\frac{1}{2}}[-24, 1, 6]^{-\frac{3}{4}}[-24, 1, 12]^{\frac{3}{2}}\\
1^{-1}2^{4}3^{-1}4^{-1}6^{2}12^{-1}&=[-3, 12, 1]^{\frac{3}{4}}[-4, 1, 1]^{-\frac{1}{4}}[-4, 1, 9]^{\frac{9}{4}}\\
1^{-1}2^{4}3^{1}4^{-1}6^{-1}12^{0}&=[-3, 24, 2]^{-\frac{3}{4}}[-3, 24, 4]^{-\frac{3}{2}}[-4, 8, 1]^{-\frac{1}{4}}[-4, 8, 3]^{-\frac{3}{2}}[-4, 8, 9]^{-\frac{9}{4}}[-8, 1, 2]^{-\frac{1}{4}}\\&[-8, 1, 4]^{\frac{1}{2}}[-8, 1, 6]^{\frac{3}{2}}[-8, 1, 12]^{-3}[-8, 1, 18]^{-\frac{9}{4}}[-8, 1, 36]^{\frac{9}{2}}[-24, 12, 1]^{\frac{3}{4}}\\
1^{0}2^{-1}3^{-2}4^{2}6^{6}12^{-3}&=[-3, 8, 2]^{\frac{1}{4}}[-3, 8, 4]^{\frac{1}{2}}[-3, 8, 6]^{\frac{3}{4}}[-3, 8, 12]^{\frac{3}{2}}[-4, 24, 1]^{\frac{1}{4}}[-4, 24, 3]^{\frac{3}{4}}\\&[-8, 12, 1]^{-\frac{1}{4}}[-8, 12, 3]^{\frac{3}{4}}[-24, 1, 2]^{\frac{1}{4}}[-24, 1, 4]^{-\frac{1}{2}}[-24, 1, 6]^{-\frac{3}{4}}[-24, 1, 12]^{\frac{3}{2}}\\
1^{0}2^{-1}3^{2}4^{2}6^{0}12^{-1}&=[-3, 8, 2]^{\frac{1}{4}}[-3, 8, 4]^{\frac{1}{2}}[-3, 8, 6]^{\frac{3}{4}}[-3, 8, 12]^{\frac{3}{2}}[-4, 24, 1]^{-\frac{1}{4}}[-4, 24, 3]^{-\frac{3}{4}}\\&[-8, 12, 1]^{\frac{1}{4}}[-8, 12, 3]^{-\frac{3}{4}}[-24, 1, 2]^{\frac{1}{4}}[-24, 1, 4]^{-\frac{1}{2}}[-24, 1, 6]^{-\frac{3}{4}}[-24, 1, 12]^{\frac{3}{2}}\\
1^{0}2^{2}3^{-2}4^{-1}6^{5}12^{-2}&=[-3, 8, 2]^{-\frac{1}{2}}[-3, 8, 4]^{-1}[-4, 24, 1]^{-\frac{1}{2}}[-8, 12, 1]^{\frac{1}{2}}[-24, 1, 2]^{-\frac{1}{2}}[-24, 1, 4]^{1}\\
1^{0}2^{2}3^{2}4^{-1}6^{-1}12^{0}&=[-3, 8, 2]^{-\frac{1}{2}}[-3, 8, 4]^{-1}[-4, 24, 1]^{\frac{1}{2}}[-8, 12, 1]^{-\frac{1}{2}}[-24, 1, 2]^{-\frac{1}{2}}[-24, 1, 4]^{1}\\
1^{1}2^{-1}3^{-1}4^{1}6^{4}12^{-2}&=[-3, 8, 2]^{-\frac{1}{4}}[-3, 8, 4]^{-\frac{1}{2}}[-3, 8, 6]^{-\frac{3}{4}}[-3, 8, 12]^{-\frac{3}{2}}[-4, 24, 1]^{-\frac{1}{4}}[-4, 24, 3]^{-\frac{3}{4}}\\&[-8, 12, 1]^{-\frac{1}{4}}[-8, 12, 3]^{\frac{3}{4}}[-24, 1, 2]^{\frac{1}{4}}[-24, 1, 4]^{-\frac{1}{2}}[-24, 1, 6]^{-\frac{3}{4}}[-24, 1, 12]^{\frac{3}{2}}\\
1^{1}2^{-1}3^{1}4^{1}6^{1}12^{-1}&=[-3, 1, 4]^{\frac{1}{2}}[-3, 1, 12]^{-\frac{3}{2}}[-3, 1, 16]^{-2}[-3, 1, 48]^{6}[-4, 12, 1]^{-\frac{1}{2}}[-4, 12, 3]^{-\frac{3}{2}}\\
1^{1}2^{1}3^{-1}4^{0}6^{2}12^{-1}&=[-3, 24, 2]^{-\frac{3}{4}}[-3, 24, 4]^{-\frac{3}{2}}[-4, 8, 1]^{\frac{1}{4}}[-4, 8, 3]^{\frac{3}{2}}[-4, 8, 9]^{\frac{9}{4}}[-8, 1, 2]^{-\frac{1}{4}}\\&[-8, 1, 4]^{\frac{1}{2}}[-8, 1, 6]^{\frac{3}{2}}[-8, 1, 12]^{-3}[-8, 1, 18]^{-\frac{9}{4}}[-8, 1, 36]^{\frac{9}{2}}[-24, 12, 1]^{-\frac{3}{4}}\\
1^{2}2^{-2}3^{-2}4^{2}6^{4}12^{-2}&=[-3, 12, 1]^{-\frac{3}{4}}[-4, 1, 1]^{-\frac{1}{4}}[-4, 1, 9]^{\frac{9}{4}}\\
1^{2}2^{-2}3^{0}4^{2}6^{1}12^{-1}&=[-3, 24, 2]^{\frac{3}{4}}[-3, 24, 4]^{\frac{3}{2}}[-4, 8, 1]^{-\frac{1}{4}}[-4, 8, 3]^{-\frac{3}{2}}[-4, 8, 9]^{-\frac{9}{4}}[-8, 1, 2]^{-\frac{1}{4}}\\&[-8, 1, 4]^{\frac{1}{2}}[-8, 1, 6]^{\frac{3}{2}}[-8, 1, 12]^{-3}[-8, 1, 18]^{-\frac{9}{4}}[-8, 1, 36]^{\frac{9}{2}}[-24, 12, 1]^{-\frac{3}{4}}\\
1^{2}2^{-1}3^{-2}4^{0}6^{5}12^{-2}&=[-3, 1, 4]^{-1}[-3, 1, 16]^{4}[-4, 12, 1]^{-1}\\
1^{2}2^{-1}3^{0}4^{0}6^{2}12^{-1}&=[-3, 8, 2]^{\frac{1}{2}}[-3, 8, 4]^{1}[-4, 24, 1]^{-\frac{1}{2}}[-8, 12, 1]^{-\frac{1}{2}}[-24, 1, 2]^{-\frac{1}{2}}[-24, 1, 4]^{1}\\
1^{-2}2^{5}4^{-2}5^{0}10^{2}20^{-1}&=[-4, 40, 1]^{\frac{1}{2}}[-8, 5, 2]^{\frac{1}{2}}[-8, 5, 4]^{1}[-20, 8, 1]^{\frac{1}{2}}[-40, 1, 2]^{-\frac{1}{2}}[-40, 1, 4]^{1}\\
1^{0}2^{2}4^{-1}5^{-2}10^{5}20^{-2}&=[-4, 40, 1]^{\frac{1}{2}}[-8, 5, 2]^{-\frac{1}{2}}[-8, 5, 4]^{-1}[-20, 8, 1]^{-\frac{1}{2}}[-40, 1, 2]^{-\frac{1}{2}}[-40, 1, 4]^{1}\\
1^{0}2^{2}4^{-1}5^{2}10^{-1}20^{0}&=[-4, 40, 1]^{-\frac{1}{2}}[-8, 5, 2]^{-\frac{1}{2}}[-8, 5, 4]^{-1}[-20, 8, 1]^{\frac{1}{2}}[-40, 1, 2]^{-\frac{1}{2}}[-40, 1, 4]^{1}\\
1^{2}2^{-1}4^{0}5^{0}10^{2}20^{-1}&=[-4, 40, 1]^{-\frac{1}{2}}[-8, 5, 2]^{\frac{1}{2}}[-8, 5, 4]^{1}[-20, 8, 1]^{-\frac{1}{2}}[-40, 1, 2]^{-\frac{1}{2}}[-40, 1, 4]^{1}\\
1^{-2}2^{5}3^{0}4^{-2}6^{0}8^{0}12^{2}24^{-1}&=[-3, 1, 16]^{-1}[-3, 1, 64]^{4}[-4, 12, 4]^{1}[-8, 24, 1]^{\frac{1}{2}}[-24, 8, 1]^{\frac{1}{2}}\\
1^{-2}2^{6}3^{0}4^{-3}6^{-1}8^{1}12^{2}24^{-1}&=[-3, 24, 1]^{\frac{3}{4}}[-8, 1, 1]^{\frac{1}{4}}[-8, 1, 3]^{-\frac{3}{2}}[-8, 1, 9]^{\frac{9}{4}}\\
1^{-1}2^{0}3^{1}4^{5}6^{-1}8^{-2}12^{0}24^{0}&=[-3, 24, 1]^{\frac{3}{4}}[-3, 24, 2]^{\frac{3}{2}}[-8, 1, 1]^{-\frac{1}{4}}[-8, 1, 2]^{\frac{1}{2}}[-8, 1, 3]^{\frac{3}{2}}[-8, 1, 6]^{-3}\\&[-8, 1, 9]^{-\frac{9}{4}}[-8, 1, 18]^{\frac{9}{2}}\\
1^{-1}2^{2}3^{1}4^{-1}6^{-1}8^{2}12^{1}24^{-1}&=[-3, 12, 8]^{-\frac{3}{4}}[-3, 12, 16]^{-\frac{3}{2}}[-4, 1, 8]^{\frac{1}{4}}[-4, 1, 16]^{-\frac{1}{2}}[-4, 1, 72]^{-\frac{9}{4}}[-4, 1, 144]^{\frac{9}{2}}\\&[-8, 8, 1]^{\frac{1}{8}}[-8, 8, 9]^{-\frac{9}{8}}[-24, 24, 1]^{\frac{3}{8}}\\
1^{-1}2^{2}3^{1}4^{0}6^{-3}8^{0}12^{5}24^{-2}&=[-3, 8, 1]^{\frac{1}{4}}[-3, 8, 2]^{\frac{1}{2}}[-3, 8, 3]^{\frac{3}{4}}[-3, 8, 6]^{\frac{3}{2}}[-24, 1, 1]^{\frac{1}{4}}[-24, 1, 2]^{-\frac{1}{2}}\\&[-24, 1, 3]^{-\frac{3}{4}}[-24, 1, 6]^{\frac{3}{2}}\\
1^{0}2^{0}3^{-2}4^{2}6^{5}8^{-1}12^{-2}24^{0}&=[-3, 1, 16]^{-1}[-3, 1, 64]^{4}[-4, 12, 4]^{-1}[-8, 24, 1]^{\frac{1}{2}}[-24, 8, 1]^{-\frac{1}{2}}\\
1^{0}2^{0}3^{2}4^{2}6^{-1}8^{-1}12^{0}24^{0}&=[-3, 1, 16]^{-1}[-3, 1, 64]^{4}[-4, 12, 4]^{-1}[-8, 24, 1]^{-\frac{1}{2}}[-24, 8, 1]^{\frac{1}{2}}\\
1^{0}2^{1}3^{-2}4^{-1}6^{4}8^{1}12^{0}24^{-1}&=[-3, 8, 1]^{\frac{1}{4}}[-3, 8, 3]^{\frac{3}{4}}[-24, 1, 1]^{-\frac{1}{4}}[-24, 1, 3]^{\frac{3}{4}}\\
1^{0}2^{1}3^{2}4^{-1}6^{-2}8^{1}12^{2}24^{-1}&=[-3, 8, 1]^{-\frac{1}{4}}[-3, 8, 2]^{-\frac{1}{2}}[-3, 8, 3]^{-\frac{3}{4}}[-3, 8, 6]^{-\frac{3}{2}}[-24, 1, 1]^{\frac{1}{4}}[-24, 1, 2]^{-\frac{1}{2}}\\&[-24, 1, 3]^{-\frac{3}{4}}[-24, 1, 6]^{\frac{3}{2}}\\
1^{1}2^{-3}3^{-1}4^{6}6^{2}8^{-2}12^{-1}24^{0}&=[-3, 24, 1]^{-\frac{3}{4}}[-8, 1, 1]^{\frac{1}{4}}[-8, 1, 3]^{-\frac{3}{2}}[-8, 1, 9]^{\frac{9}{4}}\\
1^{1}2^{-1}3^{-1}4^{0}6^{2}8^{2}12^{0}24^{-1}&=[-3, 12, 8]^{-\frac{3}{4}}[-3, 12, 16]^{-\frac{3}{2}}[-4, 1, 8]^{\frac{1}{4}}[-4, 1, 16]^{-\frac{1}{2}}[-4, 1, 72]^{-\frac{9}{4}}[-4, 1, 144]^{\frac{9}{2}}\\&[-8, 8, 1]^{-\frac{1}{8}}[-8, 8, 9]^{\frac{9}{8}}[-24, 24, 1]^{-\frac{3}{8}}\\
1^{1}2^{-1}3^{-1}4^{1}6^{0}8^{0}12^{4}24^{-2}&=[-3, 8, 1]^{-\frac{1}{4}}[-3, 8, 3]^{-\frac{3}{4}}[-24, 1, 1]^{-\frac{1}{4}}[-24, 1, 3]^{\frac{3}{4}}\\
1^{2}2^{-1}3^{0}4^{0}6^{0}8^{0}12^{2}24^{-1}&=[-3, 1, 16]^{-1}[-3, 1, 64]^{4}[-4, 12, 4]^{1}[-8, 24, 1]^{-\frac{1}{2}}[-24, 8, 1]^{-\frac{1}{2}}\\
1^{2}2^{0}3^{0}4^{-1}6^{-1}8^{1}12^{2}24^{-1}&=[-3, 24, 1]^{-\frac{3}{4}}[-3, 24, 2]^{-\frac{3}{2}}[-8, 1, 1]^{-\frac{1}{4}}[-8, 1, 2]^{\frac{1}{2}}[-8, 1, 3]^{\frac{3}{2}}[-8, 1, 6]^{-3}\\&[-8, 1, 9]^{-\frac{9}{4}}[-8, 1, 18]^{\frac{9}{2}}\\
\end{align*}
\end{appendix}

We explain the usage using an example $1^{4}2^{-2}=[-4, 1, 1]^{-2}[-4, 1, 2]^{4}$. Here $1^{4}2^{-2}$ means $\eta(\tau)^4\eta(2\tau)^{-2}$ as in Table \ref{table:etaQuotients}, $[-4, 1, 1]^{-2}$ and $[-4, 1, 2]^{4}$ means $-2E_{1}^{-4,1,1}$ and $4E_{1}^{-4,1,2}$ respectively. Thus this formal identity means $\eta(\tau)^4\eta(2\tau)^{-2}=-2E_{1}^{-4,1,1}+4E_{1}^{-4,1,2}$. Note that the weight $k$ is not shown in the notation $[D_1,D_2,t]^c$ for saving space, and it should be deduced from the eta-quotient.

\section*{Acknowledgments}
This work is supported in part by Science and Technology Commission of Shanghai Municipality (No. 22DZ2229014). The author would like to thank the anonymous referees, who provide valuable and constructive advice which substantially improves this work.

\bibliographystyle{plain}
\bibliography{ref}

\begin{thebibliography}{10}

\bibitem{AL70}
A.~O.~L. Atkin and J.~Lehner.
\newblock Hecke operators on {$\Gamma \sb{0}(m)$}.
\newblock {\em Math. Ann.}, 185:134--160, 1970.

\bibitem{AI23}
Banu~\.Irez Ayd\i~n and Ilker Inam.
\newblock On the {H}ecke eigenforms of half-integral weight and {D}edekind-eta
  products.
\newblock {\em Asian-Eur. J. Math.}, 16(7):Paper No. 2350131, 6, 2023.

\bibitem{Bha17}
Soumya Bhattacharya.
\newblock Holomorphic eta quotients of weight 1/2.
\newblock {\em Adv. Math.}, 320:1185--1200, 2017.

\bibitem{Bha21}
Soumya Bhattacharya.
\newblock Special factors of holomorphic eta quotients.
\newblock {\em Adv. Math.}, 392:Paper No. 108019, 24, 2021.

\bibitem{CKL19}
Dohoon Choi, Byungchan Kim, and Subong Lim.
\newblock Pairs of eta-quotients with dual weights and their applications.
\newblock {\em Adv. Math.}, 355:106779, 51, 2019.

\bibitem{CS17}
Henri Cohen and Fredrik Str\"omberg.
\newblock {\em Modular forms}, volume 179 of {\em Graduate Studies in
  Mathematics}.
\newblock American Mathematical Society, Providence, RI, 2017.
\newblock A classical approach.

\bibitem{DS05}
Fred Diamond and Jerry Shurman.
\newblock {\em A first course in modular forms}, volume 228 of {\em Graduate
  Texts in Mathematics}.
\newblock Springer-Verlag, New York, 2005.

\bibitem{Fin88}
Nathan~J. Fine.
\newblock {\em Basic hypergeometric series and applications}, volume~27 of {\em
  Mathematical Surveys and Monographs}.
\newblock American Mathematical Society, Providence, RI, 1988.
\newblock With a foreword by George E. Andrews.

\bibitem{GH93}
Basil Gordon and Kim Hughes.
\newblock Multiplicative properties of {$\eta$}-products. {II}.
\newblock In {\em A tribute to {E}mil {G}rosswald: number theory and related
  analysis}, volume 143 of {\em Contemp. Math.}, pages 415--430. Amer. Math.
  Soc., Providence, RI, 1993.

\bibitem{KG11}
G\"unter K\"ohler.
\newblock {\em Eta products and theta series identities}.
\newblock Springer Monographs in Mathematics. Springer, Heidelberg, 2011.

\bibitem{LO13}
Robert~J. Lemke~Oliver.
\newblock Eta-quotients and theta functions.
\newblock {\em Adv. Math.}, 241:1--17, 2013.

\bibitem{Mar96}
Yves Martin.
\newblock Multiplicative {$\eta$}-quotients.
\newblock {\em Trans. Amer. Math. Soc.}, 348(12):4825--4856, 1996.

\bibitem{MW24}
Florian M\"unkel and Kenneth~S. Williams.
\newblock Explicit formulas for the {F}ourier coefficients of a class of eta
  quotients.
\newblock {\em Hardy-Ramanujan J.}, 47:65--91, 2024.

\bibitem{Oga18}
Takeshi Ogasawara.
\newblock Some expressions for binary theta series by eta-quotients and their
  applications.
\newblock {\em J. Number Theory}, 186:493--509, 2018.

\bibitem{Ono04}
Ken Ono.
\newblock {\em The web of modularity: arithmetic of the coefficients of modular
  forms and {$q$}-series}, volume 102 of {\em CBMS Regional Conference Series
  in Mathematics}.
\newblock Conference Board of the Mathematical Sciences, Washington, DC; by the
  American Mathematical Society, Providence, RI, 2004.

\bibitem{Sage}
SageMath.
\newblock The {S}age mathematics software system (version 10.4).
\newblock {\em The Sage Developers}, 2024.

\bibitem{Wil12}
Kenneth~S. Williams.
\newblock Fourier series of a class of eta quotients.
\newblock {\em Int. J. Number Theory}, 8(4):993--1004, 2012.

\bibitem{Woh57}
Klaus Wohlfahrt.
\newblock \"{U}ber {O}peratoren {H}eckescher {A}rt bei {M}odulformen reeller
  {D}imension.
\newblock {\em Math. Nachr.}, 16:233--256, 1957.

\bibitem{YXJ13}
Olivia X.~M. Yao, Ernest X.~W. Xia, and J.~Jin.
\newblock Explicit formulas for the {F}ourier coefficients of a class of eta
  quotients.
\newblock {\em Int. J. Number Theory}, 9(2):487--503, 2013.

\bibitem{ZZ23}
Hai-Gang Zhou and Xiao-Jie Zhu.
\newblock Double coset operators and eta-quotients.
\newblock {\em J. Number Theory}, 249:537--601, 2023.

\bibitem{Zhu24}
Xiao-Jie Zhu.
\newblock Dimension formulas for modular form spaces of rational weights, the
  classification of eta-quotient characters and an extension of {M}artin's
  theorem, 2024.

\end{thebibliography}

\end{document}